\documentclass[reqno, 12pt]{amsart}
\usepackage[letterpaper,hmargin=1in,vmargin=1in]{geometry}
\usepackage{amssymb}
\usepackage{amsmath, amssymb, amsthm, verbatim, url, mathrsfs}
\usepackage{graphicx}
\usepackage[usenames,dvipsnames]{color}
\usepackage[colorlinks=true,linkcolor=Red,citecolor=Green]{hyperref}
\usepackage[labelfont=normalfont]{subcaption}

\usepackage[OT2, T1]{fontenc}
\usepackage[russian, english]{babel}

\DeclareMathOperator \re {Re}
\DeclareMathOperator \im {Im}

\numberwithin{equation}{section}

\newtheorem{prop}{Proposition}
\newtheorem{lem}[prop]{Lemma}
\newtheorem*{theorem}{Theorem}

\numberwithin{prop}{section}

\title[Nontrapping surfaces of revolution with long living resonances]{Nontrapping surfaces of revolution with long living resonances}

\author{Kiril R. Datchev}
\email{datchev@math.mit.edu}
\address{Department of Mathematics, MIT,
Cambridge, MA 02139, USA}
\email{kdatchev@purdue.edu}
\address{Department of Mathematics, Purdue University, West Lafayette, IN 47907, USA}
\author{Daniel D. Kang}
\email{ddkang@mit.edu}
\address{Department of Mathematics, MIT, Cambridge, MA 02139, USA}
\author{Andre P. Kessler}
\email{akessler@mit.edu}
\address{Department of Mathematics, MIT,
Cambridge, MA 02139, USA}

\begin{document}

\begin{abstract}
We study resonances of surfaces of revolution obtained by removing a disk from a cone and attaching a hyperbolic cusp in its place. These surfaces  include ones with nontrapping geodesic flow (every maximally extended non-reflected geodesic is unbounded) and yet infinitely many long living resonances (resonances with uniformly bounded imaginary part, i.e. decay rate). 
\end{abstract}

\maketitle

\addtocounter{section}{1}
\addcontentsline{toc}{section}{1. Introduction}
\thispagestyle{empty}

Let $a<0<b$, and let $(X,g)$ be the surface of revolution
\[
X = \mathbb{R} \times S^1, \qquad g= dr^2 + f(r)^2 d\theta^2,
\qquad
f(r) = 
\begin{cases}
	1 +ar &\mbox{if } r \leq 0, \\
	e^{-br} &\mbox{if } r > 0.
\end{cases}
\]

\begin{figure}[h]
\centering
	\begin{subfigure}{0.31\textwidth}
		\includegraphics[width=\textwidth]{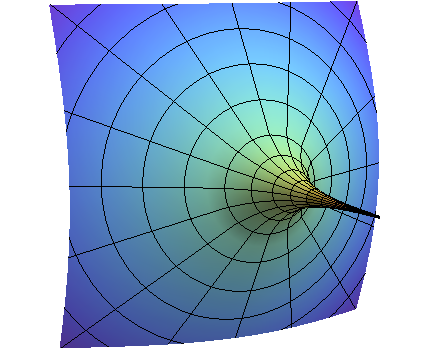}
		\caption{$a + b < 0$}
		\label{f:sumLess0}
	\end{subfigure}
	~
	\begin{subfigure}{0.31\textwidth}
		\includegraphics[width=\textwidth]{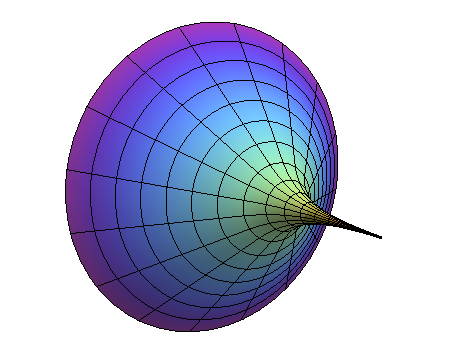}
		\caption{$a + b = 0$}
		\label{f:sumEq0}
	\end{subfigure}
	~
	\begin{subfigure}{0.31\textwidth}
		\includegraphics[width=\textwidth]{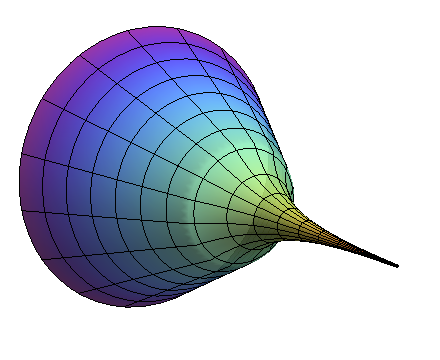}
		\caption{$a + b > 0$}
		\label{f:sumGt0}
	\end{subfigure}

\caption{The surface $(X,g)$ is obtained by attaching a conic end to a hyperbolic cusp. Three geometric pictures arise depending on the relationship between the widths of the cone and cusp. Case (b) is the most interesting, as it exhibits both nontrapping geodesic flow and long living resonances.}
\end{figure}

Let $\Delta_g$ be the nonnegative Laplacian on $(X,g)$. The resolvent $(\Delta_g - \lambda^2)^{-1}$ is holomorphic $L^2(X) \to L^2(X)$ for $\im \lambda > 0$. Poles of the continuation of its integral kernel from $\{\im \lambda > 0 \}$ to  $\{\im \lambda \le 0, \, \re \lambda >  b/2\}$ are called \textit{resonances} (see \cite{mel,zwres}).

\begin{theorem}
The surface of revolution $(X,g)$ has a sequence of resonances $(\lambda_k)_{k  \ge k_0}$ satisfying
\begin{equation}\label{e:propasy0}
\begin{split}
\re \lambda_k &= \pi b \frac{k}{\log k} \left( 1 + O\left( \frac{\log \log k}{\log k} \right) \right), \\
\im \lambda_k &= -\frac{bj}{2} + O\left( \frac{1}{\log k} \right),
\end{split}
\end{equation}

where  $j = 1$ if $a+b \neq 0$, and $j = 2$ if $a+b = 0$.
\end{theorem}
\begin{figure}[h]
\centering
\includegraphics[width=1.0\textwidth]{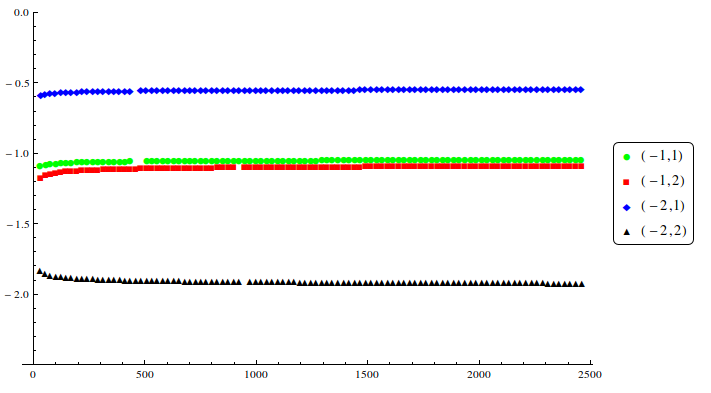}
\caption{Resonances $\lambda_k$ are plotted here in the complex plane, with $k$ ranging from $10$ to $1000$ in steps of $10$, for $(a,b)=(-1,1),(-2,1),(-1,2)$, and $(-2,2)$. They were computed by solving equation \eqref{e:derivquotient} numerically in Mathematica using FindRoot, initialized with the leading term of \eqref{e:propasy0}.}
\end{figure}	

The most interesting case is $a+b=0$, as illustrated in Figure \ref{f:sumEq0}. Then $f \in C^{1,1}$, so the geodesic flow on $(X,g)$ is well-defined and \textit{nontrapping} (see Proposition \ref{p:nontrap}) and yet there exists a sequence of \textit{long living} resonances, i.e a sequence $\lambda_k$ with $|\im \lambda_k|$ bounded. This seems to be a new phenomenon.

For many scattering problems it is known that sequences of long living resonances are impossible when the geodesic or bicharacteristic flow is nontrapping: such results go back to 
Lax and Phillips \cite{lp}, and Vainberg \cite{v} for asymptotically Euclidean scattering, and have been recently extended to asymptotically hyperbolic scattering by Vasy \cite{vasy}, Melrose, S\'a Barreto, and Vasy \cite{msv},  and Wang \cite{w}.
The result closest to the setting  of the Theorem is that of \cite{da}, where it is shown that \textit{smooth} nontrapping manifolds with cusp and funnel have no sequences of long living resonances (many more references can be found in that paper). 

Sequences of long living resonances  have been found for a variety of scattering problems, going back to work of Selberg \cite{sel} on finite volume hyperbolic quotients and Ikawa \cite{ik} on Euclidean obstacle scattering, but in each case the geodesic or bicharacteristic flow has been trapping; see e.g. Dyatlov \cite{dy} 
for some recent results and many references.
It is interesting to compare the asymptotic formula \eqref{e:propasy0} to analogous asymptotics for one dimensional potential scattering. By work of Regge \cite{regge}, Zworski \cite{z87} (see also Stepin and Tarasov \cite{stta}),  if $V\colon \mathbb R \to \mathbb R$ is supported in $[-L/2,L/2]$, is smooth away from $\pm L/2$ and vanishes to order $j \ge 0$ at $\pm L/2$, then the resonances  of $-\partial_r^2 + V(r)$  satisfy
\begin{equation}\label{e:logcurve}
\lambda_k = \frac \pi L k - i \frac{j+2}{L} \log k + O(1)
\end{equation}
in the right half plane.
Note that, as in our result, the decay rates (values of $|\im \lambda_k|$) are related to the regularity of the potential, with more regularity giving faster decay.

Sequences of resonances asymptotic to logarithmic curves (as in \eqref{e:logcurve}) have been found in other situations where the coefficients of the differential operator are not $C^\infty$: see work of Zworski \cite{z89} for scattering by radial potentials  and Burq  \cite{b} for scattering by two convex obstacles, one of which has a corner (and see also similar results by Galkowski on scattering by potentials supported on hypersurfaces in \cite{g} and in the references therein).  In \S\ref{s:3} we prove that if we  take $b<0<a$, and 
\[
X_1 = (-1/a,\infty) \times S^1, \qquad g_1= dr^2 + f_1(r)^2 d\theta^2,
\qquad
f_1(r) = 
\begin{cases}
	1 +ar &\mbox{if } r \leq 0, \\
	e^{-br} &\mbox{if } r > 0,
\end{cases}
\]
then for a fixed Fourier mode the resonances closest to the real axis obey an asymptotic more similar to \eqref{e:logcurve}, namely
\[
\lambda_k = \pi a k - \frac{i j a}{2} \log{k} + O(1).
\]
This suggests that the reflected geodesics (which, unlike the transmitted ones, may be trapped) are not by themselves enough to produce sequences of long living resonances.

In \cite{bw},  Baskin and Wunsch  study another scattering problem with nonsmooth coefficients (and in particular trapping of non-transmitted geodesics): they show that there are no long living resonances for nontrapping Euclidean scattering by a metric perturbation with cone points. For this they use the result of Melrose and Wunsch \cite{mw} that at cone points diffracted singularities for the wave equation are weaker than transmitted ones, as well as the method of Vainberg \cite{v} and Tang and Zworski \cite{tz}  which relates resolvent continuation and propagation of singularities. For a general metric $g \in C^{1, \alpha}$, $\alpha>0$, with only conormal singularities  (as in the case of a jump such as we have) de Hoop, Uhlmann, and Vasy \cite{duv} show that reflected singularities are weaker than transmitted ones: these results lead one to conjecture that there are no long living resonances for nontrapping Euclidean scattering by a metric perturbation with $C^{1,\alpha}$ jump singularities. When the discontinuities are more severe, as in the case of the transmission obstacle problem (where $g$ itself is discontinuous), long living resonances and even resonances with $|\im \lambda_k| \to 0$ have been observed (see Cardoso, Popov, and Vodev \cite{cpv} and references therein): note however that in that case, unlike in our $C^{1,1}$ case, the geodesic equation does not have unique solutions.

To prove the Theorem we use the fact that, on each Fourier mode in the angular variable, $\Delta_g$ is an ordinary differential operator and the integral kernel of $(\Delta_g- \lambda^2)^{-1}$ can be written in terms of Bessel functions. Resonances occur at values of $\lambda$ satisfying a transcendental equation (see \eqref{e:derivquotient} below), and we use Bessel function asymptotics to analyze these solutions for a fixed Fourier mode $m$. 

It would be interesting to find out if increasing the regularity of $(X,g)$ further leads to further increases in $|\im \lambda_k|$. The results of \cite{da}, that smooth nontrapping manifolds with cusp and funnel have no sequences of long living resonances, suggests that the answer is yes. To study this, one could modify $f$ on $(-R,0)$ for some $R>0$ to make  more derivatives of $f$ continuous at $0$. Then one could replace the Hankel function asymptotics used in \eqref{e:h'h} by WKB asymptotics (as in e.g. \cite[Chapter 10, \S3.1]{Olver:Asymptotics}), and compute further terms of the asymptotic expansion in Lemma \ref{l:main}.

Another interesting problem is to find \textit{all} the long living resonances of $(X,g)$. This would require asymptotics uniform in the Fourier mode $m$.  Such asymptotics for Bessel functions in terms of Airy functions are known (see e.g. \cite[\S 9.3.37]{abst}) but the resulting transcendental equation seems difficult to solve. 

\textbf{Acknowledgements.} We are very grateful to David Borthwick, David Jerison, Richard Melrose,  Andr\'as Vasy, and Maciej Zworski for many helpful discussions and suggestions. Thanks also to the two anonymous referees for their comments and corrections. KD was partially supported by the National Science Foundation under a postdoctoral fellowship.  DK and AK were partially supported by MIT under the Undergraduate Research Opportunities Program and by the National Science Foundation under grant DMS-1005944.

\section{Proof of Theorem}

\begin{lem}\label{p:nontrap}
When $a+b=0$ the manifold $(X,g)$ has a nontrapping geodesic flow.
\end{lem}

\begin{proof}
The geodesic equations of motion for this flow are
\begin{align}
\ddot{r} - f' f(\dot{\theta})^2 &= 0 \label{e:geosoln},\\
\ddot{\theta} &= 0.
\end{align}
If $f\in C^{1,1}$, there is a unique solution. Additionally, if $\dot{\theta}=0$, then $\dot{r}=0$ is disallowed because that would describe a stationary solution. In this case \eqref{e:geosoln} reduces to $\ddot{r}=0$ supplemented with the condition $\dot{r}\neq 0$, and this forces $r\rightarrow \pm\infty$ linearly.

We take $a$ negative and $b$ positive, so $f'$ is always negative and thereby from the first equation of motion \eqref{e:geosoln} we find $\ddot{r} \leq 0$. Thus as $t$ increases, $r$ must tend toward either a constant or negative infinity. In the former case, we must also have that $\ddot{r} \rightarrow 0$. However, \eqref{e:geosoln} would require that $f'f(\dot{\theta})^2 \rightarrow 0$ which is impossible if $\dot{\theta}$ does not vanish, so $r$ cannot tend to a constant. Hence $r$ tends to negative infinity and describes a nontrapping geodesic.
\end{proof}

In $(r,\theta)$ coordinates the Laplacian on $(X,g)$ is given by
\[
\Delta_g = -f^{-1} \partial_r f \partial_r - f^{-1} \partial_{\theta} f^{-1} \partial_{\theta} = - \partial_r^2 - f'f^{-1} \partial_r  - f^{-2} \partial_{\theta}^2.
\]
Let
\[
P(m) =- \partial_r^2 - f^{-1}f' \partial_r + f^{-2} m^2.
\]
On functions of the form $u(r) e^{i m \theta}$ the Laplacian acts as follows:
\[
\Delta_gu(r) e^{i m \theta} = P(m) u(r) e^{i m \theta}.
\]

\begin{lem}\label{p:mercont}
The outgoing resolvent, defined by 
\[(P(m) - \lambda^2)^{-1} \colon L^2(\mathbb R, f(r) dr) \to  L^2(\mathbb R, f(r) dr), \quad \im \lambda > 0,
\]
is holomorphic on the half plane $\im \lambda > 0$.
Its integral kernel continues meromorphically to a covering space of $\{\lambda \in \mathbb C \colon \lambda \ne 0 , \ \lambda \ne \pm b/2\}$.
\end{lem}

\begin{proof}

For $\lambda \in \mathbb C \setminus \{0\}$ and $m > 0$, the general solution to the Helmholtz equation 
 \begin{equation}\label{e:homoghelm}
(P(m) - \lambda^2)u = 0,
\end{equation}
 is given  by 
\[
u(r) =
\begin{cases}
	c_1 H^{(1)}_{m/a}(\lambda(r+1/a)) + c_2 H^{(2)}_{m/a}(\lambda(r+1/a))&\mbox{if } r \leq 0, \\ 
	c_3 e^{br/2} I_{\nu}(\frac{m}{b} e^{br}) + c_4e^{br/2} K_{\nu}(\frac{m}{b} e^{br})&\mbox{if } r > 0,
\end{cases}
\]
where $H^{(1)}$ and $H^{(2)}$ are the Hankel functions (as in \cite[\S9.1]{abst}), $I$ and $K$ are the modified Bessel functions (as in \cite[\S9.6]{abst}), $\nu$ is given by
\begin{equation}\label{e:nudef}
\nu := \sqrt{\frac 14 - \frac {\lambda^2}{b^2}} = -i \frac \lambda b\left(1 + O(\lambda^{-2})\right),
\end{equation}
  and $c_1$, $c_2$, $c_3$, $c_4 \in \mathbb C$ are  taken such that $u$ and $u'$ are continuous at $0$. Indeed, for $r \ne 0$ this follows from the differential equations solved by $H^{(1)}$, $H^{(2)}$, $I$, and $K$ (\cite[\S 9.1.1, \S 9.6.1]{abst}), and the condition at $0$ follows from the  fact that $f$ has, at worst, a jump singularity there.

By the method of variation of parameters, the inhomogeneous Helmholtz equation 
\begin{equation}\label{e:inhomoghelm}
(P(m) - \lambda^2)u = v
\end{equation}
is solved by
\begin{equation}\label{e:udefr}
u(r) = \int_{-\infty}^\infty R(r,r') v(r')dr', \quad 
R(r,r') := -\psi_1(\max\{r,r'\})\psi_2(\min\{r,r'\})/W(r'),
\end{equation}
for $\psi_1$ and $\psi_2$ linearly independent solutions of \eqref{e:homoghelm}, and $W = \psi_1\psi_2' - \psi_2\psi_1'$ their Wronskian.

The resolvent operator $v \mapsto u$ is bounded on $L^2(\mathbb R, f(r)dr)$ for $\im \lambda > 0$ if and only if 
\begin{align*}
\psi_1(r) &= e^{br/2} K_{\nu}\left(\frac{m}{b} e^{br}\right), \mbox{ } r > 0, \\
\psi_2(r) &= H^{(2)}_{m/a}(\lambda(r+1/a)), \mbox{ } r < 0,
\end{align*}
up to an overall constant factor. The condition on $\psi_1$ is justified by the asymptotics \cite[\S9.2.3, \S9.2.4]{abst} ($H^{(1)}$ grows exponentially and $H^{(2)}$ decays exponentially as $r \to - \infty$). The condition on $\psi_2$ is justified by \cite[\S9.7.1, \S9.7.2]{abst} ($I$ grows double-exponentially and $K$ decays double-exponentially as $r \to \infty$).

Meromorphic continuation now follows from the fact that the Hankel function terms are entire in $\log \lambda$ and the modified Bessel function terms are  entire in $\nu$. 
Poles of $R(r,r')$ occur at the values of $\lambda$ for which the Wronskian $W$ vanishes, or equivalently for which $\psi_1$ is a multiple of $\psi_2$, i.e. for which there is $c \in \mathbb C$ such that
\[
u(r) = \begin{cases}
	H_{m/a}^{(2)}\left(\lambda\left(r+\frac{1}{a}\right)\right)           &\mbox{if } r \leq 0 \\
	c e^{br/2} K_{\nu}\left(\frac{m}{b} e^{br}\right)  &\mbox{if } r > 0,
\end{cases}
\]
is continuous along with its first derivative at $0$. That is,
\begin{align*}
H_{m/a}^{(2)}\left(\frac{\lambda}{a}\right)           &= c K_{\nu}\left(\frac{m}{b}\right) \\
\lambda H_{m/a}^{(2)'}\left(\frac{\lambda}{a}\right)  &= c m K_{\nu}^{'}\left(\frac{m}{b}\right) + \frac{cb}{2} K_{\nu}\left(\frac{m}{b}\right).
\end{align*}
Dividing the equations to eliminate $c$ gives 
\begin{equation}\label{e:derivquotient}
\frac{ K_{\nu}'(\frac{m}{b}) }{ K_{\nu}(\frac{m}{b}) } = 
\frac{ \lambda H_{m/a}^{(2)'}(\frac{\lambda}{a}) }{ m H_{m/a}^{(2)}(\frac{\lambda}{a}) } - \frac{b}{2m}.
\end{equation}
It is important to keep track of the branch of $H^{(2)}$: as the range of $\lambda$ extends from $\{\im \lambda>0\}$ to $\{\im \lambda>0\} \cup \{|\re \lambda| > b/2\}$, 
the range of $\arg (\lambda/a)$ extends from $(-\pi,0)$ to $(-3\pi/2, \pi/2)$.

Note that no  poles (i.e. solutions to \eqref{e:derivquotient}) can occur when $\im \lambda > 0$. This follows from the self-adjointness of $\Delta_g$ on a suitable domain, but can also be checked directly from the fact that at such a pole we would have $\langle P(m)u,u\rangle_{L^2} = \lambda^2 \|u\|_{L^2}^2$
for the corresponding $u$, and the left hand side of this equation is nonnegative $\langle P(m) u, u \rangle_{L^2} \geq 0$ via integration by parts, but this is a contradiction since $\lambda^2 \not\in [0,\infty)$. This also shows that \eqref{e:udefr} is the unique $L^2$ solution to \eqref{e:inhomoghelm}  when $v \in L^2$, $\im \lambda>0$.

For $m = 0$, the analysis is similar and slightly simpler: the only change is that $I_\nu(\frac m b e^{br})$ is replaced by $e^{b\nu r}$ and $K_\nu(\frac m b e^{br})$ is replaced by $e^{-b \nu r}$.
\end{proof}

Poles of the meromorphic continuation of the integral kernel $R(r,r')$ (i.e. solutions to \eqref{e:derivquotient}) are called \textit{resonances} of $P(m)$.

\begin{prop}\label{p:main}
For any fixed  $a<0<b$, $m>0,\ \varepsilon>0,$ there are $\lambda_0>0, k_0 \in \mathbb N$ such that the resonances of $P(m)$ in the region $\{ |\lambda| > \lambda_0, |\arg \lambda| < \pi/2 - \varepsilon \}$ form a sequence $(\lambda_k)_{k \ge k_0}$ satisfying
\begin{equation}\label{e:propasy}
\begin{split}
\re \lambda_k &= \pi b \frac{k}{\log k} \left( 1 + O\left( \frac{\log \log k}{\log k} \right) \right), \\
\im \lambda_k &= -\frac{bj}{2} + O\left( \frac{1}{\log k} \right),
\end{split}
\end{equation}

where  $j = 1$ if $a+b \neq 0$, and $j = 2$ if $a+b = 0$.
\end{prop}

We begin by using Bessel function asymptotics to simplify \eqref{e:derivquotient}.

\begin{lem}\label{l:main} For fixed $a< 0 < b$ and $m>0, \ \varepsilon>0$, there are $\lambda_0>0$ and a function $r_0$ such that $\lambda \in \{ |\lambda| > \lambda_0, |\arg \lambda| < \pi/2 - \varepsilon \}$ solves \eqref{e:derivquotient} if and only if
\begin{equation}\label{e:cj}
\left(\frac m {2b}\right)^{-2\nu} \frac{\Gamma(\nu)}{\Gamma(-\nu)} = \frac{c_0}{\nu^j}(1 + r_0),
\end{equation}
 where $\nu$ is given by \eqref{e:nudef},  $j = 1$ if $a+b \neq 0$, and $j = 2$ if $a+b = 0$, and $c_0 \in \mathbb R \setminus \{0\}$. Moreover $r_0$ is holmorphic for $\lambda \in \{ |\lambda| > \lambda_0, |\arg \lambda| < \pi/2 - \varepsilon \}$ and is $O(\nu^{-1})$ there.
\end{lem}

In the proofs below it will be convenient to write 
\[
z=m/b, \qquad g(\nu) = (z/2)^{-2\nu} \frac{\Gamma(\nu)}{\Gamma(-\nu)}.
\] 

\begin{proof}[Proof of Lemma \ref{l:main}]
We first write the left hand side of \eqref{e:derivquotient} in terms of Gamma functions and remainders which are small for $|\nu|$ large.
 We recall the definitions \cite[\S9.6.10, \S9.6.2]{abst}:
\[I_\nu(z) = \frac{z^\nu}{2^\nu} \sum_{k=0}^\infty \frac{(z/2)^{2k}}{k!\Gamma(\nu+k+1)}, \qquad
K_\nu(z) = \frac{\pi}{2\sin{\pi \nu}} \left[I_{-\nu}(z)-I_\nu(z)\vphantom{\frac{a}{b}}\right].
\]
Define the remainders $R_1$, $R_2$, $R_3$, $R_4$ by the equations
\begin{equation}\label{e:inu}
I_\nu(z) = \frac{z^\nu}{2^\nu \Gamma(1+\nu)}(1+R_1), \qquad I_{-\nu}(z) = \frac{z^{-\nu}}{2^{-\nu} \Gamma(1-\nu)}(1+R_2),
\end{equation}
\begin{equation}\label{e:i-nu}
I'_\nu(z)= \frac \nu z \frac{z^\nu}{2^\nu \Gamma(1+\nu)}(1+R_3), \qquad I'_{-\nu}(z) = - \frac \nu z \frac{z^{-\nu}}{2^{-\nu} \Gamma(1-\nu)}(1+R_4), 
\end{equation}
and observe that $R_1$, $R_2$, $R_3$, and $R_4$ are all $O(\nu^{-1})$ for fixed $z$. 
For later use, we record the fact that by the recurrence relation \cite[\S9.6.26]{abst}, we have
\[
\frac{z^\nu}{2^\nu \Gamma(1+\nu)}(R_1 - R_3) = I_\nu(z) - \frac z \nu I'_\nu(z) = - \frac z \nu I_{\nu+1}(z) = - \frac{z^{\nu+2}}{2^{\nu+1} \nu \Gamma(2+\nu)}\left(1 + O\left(\frac 1\nu\right)\right),
\]
so that
\begin{equation}\label{e:inur}
R_1 - R_3 = - \frac{z^2}{2 \nu^2} + O\left(\frac 1{\nu^3}\right).
\end{equation}
We simplify the resulting formula for $K$ using the Gamma reflection formula \cite[\S6.1.17]{abst}:
\begin{align*}
K_\nu(z) &= \frac{\pi}{2\sin{\pi \nu}} \left[\left(\frac{z}{2}\right)^{-\nu}\frac{1+R_2}{\Gamma(1-\nu)} -\left(\frac{z}{2}\right)^\nu \frac{1+R_1}{\Gamma(1+\nu)}\right]\\
&= \frac{\Gamma(1-\nu)\Gamma(\nu)}{2} \left[\left(\frac{z}{2}\right)^{-\nu}\frac{1+R_2}{\Gamma(1-\nu)} -\left(\frac{z}{2}\right)^\nu \frac{1+R_1}{\nu\,\Gamma(\nu)}\right]\\
&= \frac{1}{2} \left[\left(\frac{z}{2}\right)^{-\nu}\Gamma(\nu)(1+R_2) +\left(\frac{z}{2}\right)^\nu \Gamma(-\nu)(1+R_1)\right].
\end{align*}
Similarly
\begin{align*}
K_{\nu}'(z) &= \frac{\nu}{2z} \left[\left(\frac{z}{2}\right)^\nu \Gamma(-\nu)(1 + R_3) - \left(\frac{z}{2}\right)^{-\nu} \Gamma(\nu) (1 + R_4)\right].
\end{align*}
The quotient is
\begin{equation}\label{e:k'k}
\frac{K_{\nu}'(z)}{K_{\nu}(z)} = \frac{\nu}{z}\left[\frac{1 + R_3 - g(\nu)(1 + R_4)}{1 + R_1 + g(\nu)(1 + R_2)}\right].
\end{equation}

On the other hand, by Hankel's asymptotics, the right hand side of \eqref{e:derivquotient} has an expansion in powers of $\nu$. Indeed, applying \eqref{e:happ} with $n = m/a$ and $x=\lambda/a$, and using 
\[
-i\lambda/b = \nu\left(1 - \frac 1{8\nu^2} + O(\nu^{-4})\right),
\]
(see \eqref{e:nudef}) we obtain
\begin{equation}\label{e:h'h}\begin{split}
\frac{ \lambda H_{m/a}^{(2)'}(\frac{\lambda}{a}) }{ m H_{m/a}^{(2)}(\frac{\lambda}{a}) } - \frac{b}{2m} &= \frac {-i\lambda} m\left( 1 + \frac{a+b}{2i\lambda}  + \frac{a^2 - 4 m^2}{8\lambda^2}  + O(\lambda^{-3})\right) \\
&= \frac \nu z \left( 1 - \frac{a+b}{2b \nu}  + \frac{4m^2 - b^2 - a^2}{8b^2\nu^2}  + O(\nu^{-3})\right).
\end{split}\end{equation}

Plugging \eqref{e:k'k} and \eqref{e:h'h} into \eqref{e:derivquotient} gives
\[
\frac{1 + R_3 - g(\nu)(1 + R_4)}{1 + R_1 + g(\nu)(1 + R_2)} =
1 + R_5, \qquad R_5 := - \frac{a+b}{2b \nu}  + \frac{4m^2 - b^2 - a^2}{8b^2\nu^2}  + O(\nu^{-3}).
\]
Solving for $g(\nu)$ we find, using the fact that $R_2$, $R_4$, and $R_5$ are each $O(\nu^{-1})$, that
\[
g(\nu) = \frac{1 + R_3 - (1 + R_1)(1+R_5)}{(1+R_2)(1 + R_5) + 1 + R_4} = \frac{R_3  - R_1 -  R_5 + R_1R_5}{2 + O(\nu^{-1})}.
\]
Using \eqref{e:inur} and the formula for $R_5$, we obtain
\[
g(\nu) =  \begin{cases} -(a+b + O(\nu^{-1}))/(4b\nu), & a+b \ne 0, \\ -(1 + O(\nu^{-1}))/(8\nu^2), & a+b=0. \end{cases}
\]
\end{proof}

\begin{proof}[Proof of Proposition \ref{p:main}]
By Stirling's approximation 
(see e.g. \cite[\S6.1.37]{abst}):
\begin{align*}
g(\nu) &= \left(\frac{z}{2}\right)^{-2 \nu} \frac{\Gamma(\nu)}{\Gamma(-\nu)}  
       = \left(\frac{z}{2}\right)^{-2 \nu} \frac{\sqrt{2 \pi} \nu^{\nu+1/2} e^{-\nu} \left( 1 + O\left(\frac{1}{\nu}\right) \right)}{\sqrt{2 \pi} (-\nu)^{-\nu+1/2} e^{\nu} \left( 1 + O\left(\frac{1}{\nu}\right) \right)}  \\
       &= \left(\frac{e z}{2}\right)^{-2 \nu} \nu^{\nu + 1/2} (-\nu)^{\nu - 1/2}  \left( 1 + O\left(\frac{1}{\nu}\right) \right).
\end{align*}

Using $\log$ to denote the principal branch of the logarithm, we have
\begin{align*}
\log{g(\nu)} &= \left(\nu + \frac{1}{2}\right) \log{\nu} + \left(\nu - \frac{1}{2}\right) \log{(-\nu)} - 2\nu\log{\frac{ez}{2}} + O\left(\frac{1}{\nu}\right) .\\
\end{align*} 
Plugging this into \eqref{e:cj}, we obtain
\begin{align*}
\left(\nu + \frac{1}{2}\right) \log{\nu} + \left(\nu - \frac{1}{2}\right) \log{(-\nu)} - 2\nu\log{\frac{ez}{2}} &= \tilde{c}_0 - j \log{\nu} - 2 \pi i k + O\left(\frac{1}{\nu}\right),
\end{align*}
where $\exp(\tilde{c}_0)=c_0$, and $k \in \mathbb Z$. 
With $\tilde{\nu} = i(2\nu + j)/ez$ this becomes
\begin{equation}\label{e:nutilde2}
 \frac{2 \pi k}{e z}     =  \tilde{\nu} \log{\tilde{\nu}}  + r(\tilde \nu), 
\end{equation}
where $r(\tilde \nu)$ is holomorphic, independent of $k$, and bounded as $ \re \tilde{\nu} \to +\infty$ (since $\re \lambda \to + \infty$ implies $\im \nu \to - \infty$ and hence $\re \tilde \nu \to + \infty$). Indeed, since $\im \nu <0$ (because $\re \lambda >b/2)$, we have $\log{(-\nu)} = \log{\nu} + \pi i$ and hence
\begin{align*}
\tilde{c}_0 - 2 \pi i k
&= 2 \nu \log{\nu} + j \log{\nu} + \left(\nu - \frac{1}{2}\right) i \pi - 2\nu\log{\frac{ez}{2}} + O\left(\frac{1}{\nu}\right)\\
&= \left(2 \nu + j\right) \left(\log{\nu} + \frac{\pi i}{2} - \log{\frac{ez}{2}}\right) + j \log{\frac{ez}{2}} - i\pi \frac{j+1}{2} + O\left(\frac{1}{\nu}\right)\\
&= \left(2 \nu + j\right) \left(\log{(2\nu + j)} - \frac{1}{2\nu + j} + \frac{\pi i}{2} - \log{ez}\right) - i\pi \frac{j+1}{2} + j \log{\frac{ez}{2}} + O\left(\frac{1}{\nu}\right)\\
&= \left(2 \nu + j\right) \left(\log{(i(2\nu + j))} - \log{ez}\right) - i\pi \frac{j+1}{2} +\log{\frac{ez}{2}} - 1  + O\left(\frac{1}{\nu}\right)\\
&= \left(2 \nu + j\right) \left(\log{ \frac{i(2\nu + j)}{ez} } \right) - i\pi \frac{j+1}{2} +\log{\frac{ez}{2}} - 1  + O\left(\frac{1}{\nu}\right),\\
\end{align*}
giving \eqref{e:nutilde2}.
 We will show that if $\lambda_0$ and $k$ are large enough, then \eqref{e:nutilde2} has a unique solution $\tilde \nu_k$ corresponding to a $\lambda_k$ in $\{\lambda \in \mathbb C \colon \re \lambda \ge \lambda_0\}$, and we will compute its asymptotics. Changing variables back to $\lambda$ will give \eqref{e:propasy}.

For this we use the Lambert W function. Recall that for $\re \zeta >0$, the equation 
\begin{equation}\label{e:lambert}
\tilde \nu \log \tilde \nu = \zeta
\end{equation}
has a unique solution with $\re \log \tilde \nu  >1$ and $|\im \log \tilde \nu| < \pi/2$ (and hence a unique solution with $\re \tilde \nu >1$), and it is given by the principal branch of the Lambert W function (see \cite[Figure 4 and Figure 5]{cghjk}). As $\re \zeta \to \infty$, it obeys (see \cite[(4.20)]{cghjk})
\begin{equation}\label{e:lambertasy}
\tilde \nu(\zeta) = \frac{\zeta}{\log\zeta}\left(1 + O\left(\frac{\log\log\zeta}{\log\zeta}\right)\right).
\end{equation}
For $k$ sufficiently large, this reduces solving  \eqref{e:nutilde2} with $\re \tilde \nu >1$ to solving
\[
 \frac{2 \pi k}{e z}     = \zeta  + r(\tilde \nu(\zeta))
\]
for $\zeta$ with $\re \zeta>0$. But e.g. \cite[Chapter 1, Theorem 5.1]{Olver:Asymptotics} guarantees that such a solution exists and is unique, provided $k$ is sufficiently large. We denote the solution by $\zeta_k$, and the corresponding $\tilde \nu$ by $\tilde \nu_k$.

Then \eqref{e:lambertasy} becomes
\[
\log \tilde \nu_k = \log\left(\frac{2 \pi k}{e z} + O(1)\right) - \log\log\left(\frac{2 \pi k}{e z} + O(1)\right) + O\left(\frac{\log\log\left(\frac{2 \pi k}{e z} + O(1)\right)}{\log\left(\frac{2 \pi k}{e z} + O(1)\right)}\right),
\]
that is to say
\begin{align}\label{e:nuk1}
\tilde{\nu}_k &= \frac{2\pi}{ez} \frac{k}{\log{k}} \left(1 + O\left( \frac{\log{\log{k }}}{ \log{ k} } \right) \right). 
\end{align}

To obtain more precise asymptotics for $\im \tilde{\nu}$, we take the imaginary part of \eqref{e:nutilde2}:
\[
  \re \tilde\nu_k \arg{\tilde\nu_k} + \im \tilde\nu_k \log|\tilde\nu_k|= O(1) \label{e:nutilde}
\]
Note that, since $\re \tilde \nu_k > 0$, both terms on the left hand side have the same sign. Consequently $\im \tilde\nu_k \log|\tilde\nu_k|= O(1)$ and so
\[
\im \tilde\nu_k = O\left(\frac 1 {\log k}\right).
\]

Recalling that in terms of $\lambda$ we have
\[
\lambda_k = \frac{b \tilde{\nu}_k e z - i b j}{2} \left( 1 + O\left( \frac{1}{\tilde{\nu}_k^2} \right) \right),
\]
we conclude, using $z=m/b$, that
\begin{align*}
\re \lambda_k &= \frac{bez}{2} \re \tilde{\nu}_k \left( 1 + O\left( \frac{1}{\tilde{\nu}_k^2} \right) \right) 
            = \pi b  \frac{k}{\log k} \left( 1 + O\left( \frac{\log \log k}{\log k} \right) \right), \\
\im \lambda_k &= \left( \frac{bez}{2} \im \tilde{\nu}_k - \frac{bj}{2} \right) \left( 1 + O\left( \frac{1}{\tilde{\nu}_k^2} \right) \right) 
            = -\frac{bj}{2} + O\left( \frac{1}{\log k} \right).
\end{align*}

\end{proof}

This completes the proof of the Theorem. To interpret \eqref{e:propasy} as a Weyl asymptotic, put
\[
N(\lambda) = \# \left\{ k : \lambda_0 \le \re \lambda_k \leq \lambda\right\}.
\]
Then, if $W=-\log k$,
\begin{align*}
\lambda_k &= \pi b \frac{k}{\log{k}} \left(1+ r(k)\right)= -\frac{\pi b}{We^{W}} \left(1+r\right).
\end{align*}
We obtain from this that $We^W = - \pi b \left(1+r\right)/\lambda_k$, and thus, by \eqref{e:lambertasy},
\begin{align*}
W &= \log{\left(\frac{\pi b}{\lambda_k}(1+r)\right)}-\log{\left(-\log{\left(\frac{\pi b}{\lambda_k}(1+r) \right)}\right)}+O\left(\frac{\log{\left(-\log{\left(\frac{\pi b}{\lambda_k}(1+r) \right)}\right)}}{\log{\left(\frac{\pi b}{\lambda_k}(1+r)\right)}}\right)\\
&= -\log{\lambda_k} +\log{\pi b} + O(r) -\log{\log{\lambda_k}}+O\left(\frac{1}{\log{\lambda_k}}\right)+O\left(\frac{\log\log{\lambda_k}}{\log{\lambda_k}}\right).
\end{align*}
From this we get the result
\[
k = e^{-W} = \frac{\lambda_k \log \lambda_k}{\pi b } \left(1+O\left(\frac{\log\log{\lambda_k}}{\log{\lambda_k}}\right)\right),
\]
and hence
\[
N(\lambda) = \frac{\lambda \log \lambda}{\pi b} \left(1+O\left(\frac{\log\log{\lambda}}{\log{\lambda}}\right)\right).
\]
On the other hand
\[
\frac 1 {2\pi} \operatorname{Vol} \{\rho \in \mathbb R, r \ge 0 \colon \rho^2 + m^2 e^{2br} \le \lambda^2\} = \frac {\lambda \log \lambda} {\pi b} + O(\lambda),
\]
allowing us to interpret our result as a Weyl asymptotic for $N(\lambda)$: note  this is the same as the asymptotic obeyed by the eigenvalues of $- \partial_r^2 + m^2e^{2br}$ on $(0,\infty)$: see e.g. \cite[(7.4.2)]{tie}.

\section{Nontrapping surfaces with  funnel} \label{s:3}

If instead $b<0<a$ and $(X_1,g_1)$ is given by
\[
X_1 = (-1/a,\infty) \times S^1, \qquad g_1= dr^2 + f_1(r)^2 d\theta^2,
\qquad
f_1(r) = 
\begin{cases}
	1 +ar &\mbox{if } r \leq 0, \\
	e^{-br} &\mbox{if } r > 0,
\end{cases}
\]
then $(X_1,g_1)$ is a surface of revolution with a cone point and a funnel. In the case $a=1$ the cone point disappears and we have a surface of revolution with a disk and a funnel.
Let
\[
P_1(m) =- \partial_r^2 - f_1^{-1}f_1' \partial_r + f_1^{-2} m^2.
\]
On functions of the form $u(r) e^{i m \theta}$ the Laplacian acts as
\[
\Delta_{g_1}u(r) e^{i m \theta} = P_1(m) u(r) e^{i m \theta}.
\] 

\begin{prop}\label{p:main2}
For any $m>0$ there are $\lambda_0>0, k_0 \in \mathbb N$ such that the resonances of $P_1(m)$ in the region  $\{ \re \lambda \ge \lambda_0\}$ form a sequence $(\lambda_k)_{k \ge k_0}$ satisfying
\[
\lambda_k = \pi a k - \frac{i j a}{2} \log{k} + O(1)
\]
where  $j = 1$ if $a+b \neq 0$, and $j = 2$ if $a+b = 0$.
\end{prop}

\begin{proof}
Arguing as in the proof of Proposition \ref{p:mercont},  with the difference that the asymptotics
 \cite[\S9.2.3, \S9.2.4]{abst} for the Hankel functions must now be replaced by \cite[\S9.1.7, \S9.1.8, \S9.1.9]{abst}, we find that resonances are given by solutions to
\begin{equation} \label{e:ij}
\frac{ I_{\nu}'(\frac{m}{b}) }{ I_{\nu}(\frac{m}{b}) } = 
\frac{ \lambda J_{m/a}'(\frac{\lambda}{a}) }{ m J_{m/a}(\frac{\lambda}{a}) } - \frac{b}{2m},
\end{equation}
where now \eqref{e:nudef} is replaced by
\[
\nu = \sqrt{\frac 14 - \frac {\lambda^2}{b^2}} = i \frac \lambda b\left(1 - \frac{b^2}{8\lambda^2} + O\left(\frac 1 {\lambda^{4}}\right)\right),
\]
so that again $\re\nu>0$ when $\im \lambda >0$. (We are using here the Friedrichs extension of the Laplacian at the cone point -- see e.g. \cite[(3.3)]{mw}).
 
Using \eqref{e:inu}, \eqref{e:i-nu}, and \eqref{e:inur}, the left hand side of \eqref{e:ij}  becomes
\[
\frac{ I_{\nu}'(z) }{ I_{\nu}(z) } =   \frac \nu z \frac{(1+R_3)}{(1+R_1)} =  \frac{\nu}{z} + \frac{z}{2 \nu} + O\left( \frac{1}{\nu^2} \right),
\]
giving
\[
\frac{ I_{\nu}'(\frac{m}{b}) }{ I_{\nu}(\frac{m}{b}) } = 
\frac{i \lambda}{m} - \frac{i (b^2 + 4m^2)}{8 m \lambda} + O\left( \frac{1}{\lambda^2} \right),
\]
so \eqref{e:ij} becomes
\[
\frac{ J_{m/a}'(\frac{\lambda}{a}) }{  J_{m/a}(\frac{\lambda}{a}) } = 
i + \frac b {2\lambda} -  \frac{i (b^2 + 4m^2)}{8  \lambda^2} + O\left( \frac{1}{\lambda^3} \right) =: i +j_0.
\]
Plugging this $j_0$ into \eqref{e:chiapp}, with $x = \lambda/a$ and $n = m/a$, and simplifying gives
\[
e^{-2i\chi} = \begin{cases} (i(a+b) + O(\lambda^{-1}))/(4\lambda), & a + b \ne 0, \\ (a^2 + O(\lambda^{-1}))/(8\lambda^2), & a + b = 0,\end{cases} 
\]
where $\chi = \frac \lambda a  - \frac{\pi m}{2a} - \frac \pi 4$. As in the previous section, we write 
\[
e^{-2i \chi} = c_0 \lambda^{-j} (1 + O(\lambda^{-1})).
\] 

We solve for $\lambda$ by taking the log of both sides:
\begin{align*}
-2i \chi - 2 \pi k i &= \log{c_0} - j \log{\lambda} + O(\lambda^{-1})\\
-2i \left( \frac{\lambda}{a} - \frac{\pi m}{2a} - \frac{\pi}{4} \right) - 2 \pi ki &=  \log{c_0} - j \log{\lambda} + O(\lambda^{-1})\\
\frac{2i}{a} \lambda - j \log{\lambda} &=  - 2 \pi k i  + O(1).\\
\end{align*}

Substituting $\tilde{\lambda} = \frac{-2i}{ja} \lambda$, we get
\begin{align*}
\tilde{\lambda} + \log{ \frac{\tilde{\lambda} j a i}{2} } &= \frac{2\pi ki}{j}  + O(1)\\ 
\tilde{\lambda} + \log{\tilde{\lambda}} &= \frac{2\pi ik}{j} + O(1).
\end{align*}

Using the Lambert W function to solve this (as we solved \eqref{e:nutilde2} above), we see that 
\[
\tilde{\lambda}_k = -\frac{2\pi ik}{j} - \log{k} + O(1), \qquad
\lambda_k = \pi a k - \frac{i j a}{2} \log{k} + O(1).
\]
\end{proof}

\appendix
\section{Hankel's asymptotics for Bessel functions}

Here we collect some consequences of Hankel's asymptotics for $H^{(2)}$ and $J$. Let $n \in \mathbb C$ be fixed. For $x$ in the universal cover of $\mathbb C \setminus \{0\}$, as $|x| \to \infty$ with $\arg x$ varying in a compact subset of $(-2\pi,\pi)$, by \cite[\S 9.2.8--16]{abst} we have
\[
\frac{H^{(2)'}_n(x)}{H^{(2)}_n(x)} = \frac{-i R(n,x) - S(n,x)}{P(n,x) - i Q(n,x)},
\]
where
\[\begin{split}
P(n,x) = 1 - \frac{(4n^2-1)(4n^2-9)}{128x^2} + O(x^{-4}), \qquad Q(n,x) = \frac{4n^2-1}{8x} + O(x^{-3}), \\
R(n,x) = 1 - \frac{(4n^2-1)(4n^2+15)}{128x^2} + O(x^{-4}), \qquad S(n,x) = \frac{4n^2 + 3}{8x} + O(x^{-3}).
\end{split}\]
Consequently,
\begin{equation}\label{e:happ}\begin{split}
\frac{H^{(2)'}_n(x)}{H^{(2)}_n(x)} &= \frac{-i - \frac{4n^2 + 3}{8x} +i \frac{(4n^2-1)(4n^2+15)}{128x^2} + O(x^{-3})}{1 - i \frac{4n^2-1}{8x} -\frac{(4n^2-1)(4n^2-9)}{128x^2} + O(x^{-3})} \\
&= -i - \frac 1 {2x}   + i \frac{4n^2-1}{8x^2} +  O(x^{-3}). 
\end{split}\end{equation}
On the other hand, for $x \in \mathbb C$,   as $|x| \to \infty$ with $\arg x$  varying in a compact subset of $(-\pi,\pi)$, by \cite[\S 9.2.5, \S 9.2.11]{abst} we have
\[
\frac{J'_n(x)}{J_n(x)} = \frac{-R \sin \chi - S \cos \chi}{P \cos \chi - Q \sin \chi},
\]
with $P,\ Q,\ R,\ S$ as above and $\chi = x - (2n+1)\pi/4$. Rearranging terms we find that
\[
e^{-2i\chi} = \frac{\frac{J'_n(x)}{J_n(x)}(Q-Pi) - R - Si}{\frac{J'_n(x)}{J_n(x)}(Q + Pi) - R + Si}.
\]
Now if $J'_n(x)/J_n(x) = i + j_0$, with $j_0 = O(x^{-1})$, then
\[
\frac{J'_n(x)}{J_n(x)}(Q + Pi) - R + Si = -2 + O(x^{-1}),
\]
and
\[
\frac{J'_n(x)}{J_n(x)}(Q-Pi) - R - Si =   \frac {-i} {2x} + \frac{3(4n^2-1)}{16x^2} + j_0\left(-i + \frac{4n^2-1}{8x}\right) + O(x^{-3}),
\]
giving
\begin{equation}\label{e:chiapp}
e^{-2i\chi} = -\frac 1 2  \left( \frac {-i} {2x} + \frac{3(4n^2-1)}{16x^2} + j_0\left(-i + \frac{4n^2-1}{8x}\right) + O(x^{-3})\right) (1 + O(x^{-1})).
\end{equation}

\def\arXiv#1{\href{http://arxiv.org/abs/#1}{arXiv:#1}}

\end{document}